\documentclass[12pt,reqno]{amsart}

\newcommand\version{November 11, 2024}

\usepackage{amsmath,amsfonts,amsthm,amssymb,amsxtra}
\usepackage{bbm} 
\usepackage{dsfont}
\usepackage[hidelinks]{hyperref} 
\usepackage{color}
\usepackage{comment}
\usepackage[numbers, sort]{natbib}
\usepackage{enumitem}
\usepackage{moreenum}
\usepackage[greek,english]{babel}

\newtheorem{theorem}{Theorem}
\newtheorem{proposition}[theorem]{Proposition}
\newtheorem{lemma}[theorem]{Lemma}
\newtheorem{corollary}[theorem]{Corollary}

\theoremstyle{definition}

\newtheorem{assumption}[theorem]{Assumption}

\theoremstyle{remark}

\newtheorem{remark}[theorem]{Remark}

\newcommand{\1}{\mathds{1}}

\renewcommand{\epsilon}{\varepsilon}

\newcommand{\N}{\mathbb{N}}

\renewcommand{\phi}{\varphi}
\newcommand{\R}{\mathbb{R}}

\newcommand{\Haus}{\mathcal{H}}

\DeclareMathOperator{\Tr}{Tr}

\newcommand{\limplus}{{\mathchoice{\vcenter{\hbox{$\scriptstyle +$}}}
  {\vcenter{\hbox{$\scriptstyle +$}}}
  {\vcenter{\hbox{$\scriptscriptstyle +$}}}
  {\vcenter{\hbox{$\scriptscriptstyle +$}}}
}}
\newcommand{\limminus}{{\mathchoice{\vcenter{\hbox{$\scriptstyle -$}}}
  {\vcenter{\hbox{$\scriptstyle -$}}}
  {\vcenter{\hbox{$\scriptscriptstyle -$}}}
  {\vcenter{\hbox{$\scriptscriptstyle -$}}}
}}
\newcommand{\limpm}{{\mathchoice{\vcenter{\hbox{$\scriptstyle \pm$}}}
  {\vcenter{\hbox{$\scriptstyle \pm$}}}
  {\vcenter{\hbox{$\scriptscriptstyle \pm$}}}
  {\vcenter{\hbox{$\scriptscriptstyle \pm$}}}
}}

\begin{document}

\title[\version]{Spectral asymptotics for Robin Laplacians\\ on Lipschitz sets}

\author{Rupert L. Frank}
\address[Rupert L. Frank]{Mathe\-matisches Institut, Ludwig-Maximilians Universit\"at M\"unchen, The\-resienstr.~39, 80333 M\"unchen, Germany, and Munich Center for Quantum Science and Technology, Schel\-ling\-str.~4, 80799 M\"unchen, Germany, and Mathematics 253-37, Caltech, Pasa\-de\-na, CA 91125, USA}
\email{r.frank@lmu.de}

\author{Simon Larson}
\address{\textnormal{(Simon Larson)} Mathematical Sciences, Chalmers University of Technology and the University of Gothenburg, SE-41296 Gothenburg, Sweden}
\email{larsons@chalmers.se}

\newcommand\blfootnote[1]{%
  \begingroup
  \renewcommand\thefootnote{}\footnote{#1}%
  \addtocounter{footnote}{-1}%
  \endgroup
}

\blfootnote{\copyright\, 2024 by the authors. This paper may be reproduced, in its entirety, for non-commercial purposes.\\
	Partial support through US National Science Foundation grant DMS-1954995 (R.L.F.), the German Research Foundation grants EXC-2111-390814868 and TRR 352-Project-ID 470903074 (R.L.F.), the Knut and Alice Wallenberg foundation grant KAW 2017.0295 (S.L.), as well as the Swedish Research Council grant no.~2023-03985 (S.L.) is acknowledged.}

\dedicatory{Dedicated to the memory of Dimitri Yafaev}

\begin{abstract}
	We prove two-term spectral asymptotics for the Riesz means of the eigenvalues of the Laplacian on a Lipschitz domain with Robin boundary conditions. The second term is the same as in the case of Neumann boundary conditions. This is valid for Riesz means of arbitrary positive order. For orders at least one and under additional assumptions on the function determining the boundary conditions we derive leading order asymptotics for the difference between Riesz means of Robin and Neumann eigenvalues.
\end{abstract}

\maketitle

\section{Introduction and main results}

The present paper is a companion to our recent works \cite{FrankLarson_24,FrankLarson_24b}, but can be read independently. In \cite{FrankLarson_24} we have studied two-term asymptotics for the Riesz means of eigenvalues of the Dirichlet and Neumann Laplacians. Our focus was to derive such asymptotics for small Riesz exponents, in fact, for any positive exponent, and under weak regularity assumptions on the underlying domain, which we assumed to be a bounded open set with Lipschitz boundary. Our goal in the present paper is to extend these results to the case of Robin boundary conditions.

To be more specific, let $\Omega\subset \R^d$, $d\geq 2$, be an open and bounded set with Lipschitz regular boundary. For $\sigma \colon \partial\Omega \to \R$ we consider the Robin Laplace operator $-\Delta_\Omega^{(\sigma)}$ in $L^2(\Omega)$ defined via the quadratic form
\begin{equation*}
	H^1(\Omega) \ni u \mapsto \int_\Omega |\nabla u(x)|^2\,dx + \int_{\partial\Omega} \sigma(x) |u(x)|^2\,d\Haus^{d-1}(x)\,.
\end{equation*}
Here $\mathcal H^{d-1}$ denotes $(d-1)$-dimensional Hausdorff measure, whose restriction to $\partial\Omega$ coincides with the usual surface measure. The Lebesgue spaces $L^p(\partial\Omega)$ are defined with respect to this measure.

We always assume that $\sigma\in L^p(\partial\Omega)$ for some $p>1$ if $d=2$ and $p\geq d-1$ if $d>2$. It then follows from Sobolev trace inequalities, which are valid in view of the extension property of $\Omega$, that the above quadratic form is lower semibounded and closed. Since the quadratic form $u\mapsto \int_{\partial\Omega} \sigma |u|^2 \,d\mathcal H^{d-1}$ is compact in $H^1(\Omega)$, it follows from Rellich's compactness theorem that the spectrum of the operator $-\Delta_\Omega^{(\sigma)}$ is discrete, that is, it consists of eigenvalues of finite multiplicities accumulating only at infinity. Our interest is in the asymptotic behavior of these eigenvalues.

As is well known, on functions in its domain (and, in particular, on eigenfunctions) the operator $-\Delta_\Omega^{(\sigma)}$ acts as the Laplacian and such functions $u$ satisfy the Robin boundary condition $\partial_\nu u +\sigma u=0$ on $\partial\Omega$ in a generalized sense. Here $\nu$ denotes the outer unit normal and $\partial_\nu$ the corresponding derivative.

The following is our first main result. It provides two-term spectral asymptotics for Riesz means of arbitrary positive order. Both terms in the asymptotics are the same as in the Neumann case and, in particular, are independent of the function $\sigma$ describing the Robin boundary condition. The constants appearing there are given by
$$
L_{\gamma,d}^{\rm sc} = (4\pi)^{-\frac d2}\ \frac{\Gamma(\gamma+1)}{\Gamma(\gamma+\frac d2+1)} \,.
$$
In what follows we use the notation $x_\limpm = \frac{1}{2}(|x|\pm x)$.
\begin{theorem}\label{main}
	Let $d\geq 2$, let $\Omega\subset\R^d$ be a bounded open set with Lipschitz boundary and let $\sigma\in L^p(\partial\Omega)$ for some $p>d-1$. Then for any $\gamma>0$, as $\lambda\to\infty$,
	$$
	\Tr(-\Delta_\Omega^{(\sigma)}-\lambda)_\limminus^\gamma = L_{\gamma,d}^{\rm sc} |\Omega| \lambda^{\gamma+\frac d2} + \frac14 L_{\gamma,d-1}^{\rm sc} \mathcal H^{d-1}(\partial\Omega) \lambda^{\gamma+\frac{d-1}2}+ o(\lambda^{\gamma+\frac{d-1}2}) \,.
	$$
\end{theorem}

The assumption $p>d-1$ in the theorem is optimal in the $L^p$-scale when $d=2$. When $d>2$, it misses the borderline case $p=d-1$ and we do not know whether the asymptotics remain valid in this case.

The novelty of this result is threefold. First, we can reach arbitrarily small, positive orders $\gamma$; second, the Lipschitz assumption on the boundary is a rather minimal requirement in order to have the second term in the asymptotics well defined; third, the integrability assumptions on $\sigma$ are rather weak. Under stronger assumptions, results of this type are already known. For instance, in \cite{FrankGeisinger12} the corresponding result is proved for $\gamma\geq 1$, $C^1$-boundary and continuous $\sigma$. Moreover, if $\Omega$ and $\sigma$ are smooth and if, in addition, a certain dynamical condition is satisfied, then the asymptotics hold even for $\gamma=0$, as shown by Ivrii \cite{Ivrii80}.

A consequence of Theorem~\ref{main} is a corresponding two-term asymptotic expansion for the trace of the associated heat kernel,
\begin{equation*}
	\Tr(e^{t\Delta_\Omega^{(\sigma)}}) = (4\pi t)^{-\frac{d}2}\Bigl(|\Omega| + \frac{\sqrt{\pi t}}2 \Haus^{d-1}(\partial\Omega)+ o(\sqrt{t})\Bigr) \quad \mbox{as }t \to 0^\limplus\,.
\end{equation*}
Again, under stronger assumptions on the smoothness of $\partial\Omega$ and on $\sigma$ it is known that the heat trace has an asymptotic expansion even to higher orders in $\sqrt{t}$; see, e.g., \cite{BransonGilkey_90}. However, when $\partial\Omega$ is merely Lipschitz one cannot generally expect to have any further terms in the asymptotic expansion of $\Tr(e^{t\Delta^{(\sigma)}})$. Indeed, according to Brown~\cite{Brown93} the error term $o(\sqrt{t})$ in the asymptotic expansion for the Neumann problem is sharp on the algebraic scale; see also \cite{FrankLarson_JMP20} for a corresponding result for the Dirichlet Laplacian. We shall see in Section \ref{sec: Proofs asymptotics of differences} below that the same conclusion holds for Robin Laplace operators under weak assumptions on $\sigma$.

We thus find it rather remarkable that under fairly weak assumptions on $\sigma$ the difference of the traces of heat kernels as well as differences of Riesz means of two different Robin Laplace operators obey power-like asymptotic laws. This is our second main result.
\begin{theorem}\label{thm: Difference of Riesz means intro}
	Let $\Omega \subset \R^d$ be open, bounded, with Lipschitz regular boundary. Assume that $\sigma \in L^p(\partial\Omega)$ for some $p>2(d-1)$ and that $\sigma_\limminus\in L^\infty(\partial\Omega)$. Then for any $\gamma \geq 1$, as $\lambda \to \infty$,
	\begin{equation*}
		\Tr(-\Delta_\Omega^{(0)}-\lambda)_\limminus^\gamma-\Tr(-\Delta_\Omega^{(\sigma)}-\lambda)_\limminus^\gamma  = \frac{L_{\gamma,d-2}^{\rm sc}}{2\pi} \int_{\partial\Omega}\sigma(x)\,d\Haus^{d-1}(x) \lambda^{\gamma+ \frac{d-2}{2}}+ o(\lambda^{\gamma+ \frac{d-2}{2}})\,.
	\end{equation*} 
\end{theorem}

The assumptions on $\sigma$ can actually be relaxed, at least to some degree. In Section~\ref{sec: Proofs asymptotics of differences} we shall prove a version of Theorem \ref{thm: Difference of Riesz means intro} under the assumption that the integral kernel of $e^{t\Delta_\Omega^{(\sigma)}}$ satisfies a Gaussian upper bound.

The assumption $\gamma\geq 1$ in Theorem \ref{thm: Difference of Riesz means intro} (in contrast to the weaker assumption $\gamma>0$ in Theorem \ref{main}) is needed for our use of a Tauberian theorem. It guarantees that $\lambda\mapsto \Tr(-\Delta_\Omega^{(\sigma')}-\lambda)_\limminus^\gamma-\Tr(-\Delta_\Omega^{(\sigma)}-\lambda)_\limminus^\gamma$ is monotone when $\sigma'\geq \sigma$. A variant of this technique appeared in our recent paper \cite{FrankLarson_24}.

As shall be proved in Section~\ref{sec: Proofs asymptotics of differences}, Theorem~\ref{thm: Difference of Riesz means intro} with $\gamma =1$ is equivalent to the following result concerning the asymptotic behavior of the average of gaps between the Robin and Neumann eigenvalues. We enumerate by $\lambda_n(-\Delta_\Omega^{(\sigma)})$, $n\geq 1$, the eigenvalues of $-\Delta_\Omega^{(\sigma)}$ in nondecreasing order with multiplicities taken into account.

\begin{corollary}\label{cor: gap asymptotics intro}
    Let $\Omega \subset \R^d$ be open, bounded, with Lipschitz regular boundary. Assume that $\sigma \in L^p(\partial\Omega)$ for some $p>2(d-1)$ and that $\sigma_\limminus \in L^\infty(\partial\Omega)$. Then, as $N \to \infty$, 
	\begin{equation*}
		\sum_{n=1}^N \frac{\lambda_n(-\Delta_\Omega^{(\sigma)})-\lambda_n(-\Delta_\Omega^{(0)})}{N} = \frac{2}{|\Omega|}\int_{\partial\Omega}\sigma(x)\,d\Haus^{d-1}(x) + o(1)\,.
	\end{equation*} 
\end{corollary}

The asymptotics in Corollary~\ref{cor: gap asymptotics intro} were recently obtained in \cite{Rudnick_etal_21} under the assumption that $\Omega$ has piecewise smooth boundary and that $\sigma$ is positive and continuous. Their proof relies on a rather deep pointwise Weyl law from \cite{HassellZelditch_04}. Independently of the proof of Corollary \ref{cor: gap asymptotics intro} we obtain a variant of a pointwise Weyl law under weaker assumptions on $\Omega$ and $\sigma$ by more elementary means; see Theorem \ref{density}. We mention in passing that, besides the mean gap asymptotics, the paper \cite{Rudnick_etal_21} contains also very interesting results concerning the individual gap asymptotics under additional assumptions.

\medskip

We are grateful to the anonymous referee for their helpful remarks.

\medskip

We would like to dedicate this article to the memory of Dimitri Yafaev. His deep contributions to spectral theory and his clear style of mathematical writing had a profound influence on us. He is sorely missed.


\section{Two-term asymptotics}

Our goal in this section is to prove Theorem~\ref{main}. Our proof relies on the following bound in the spirit of Cwikel--Lieb--Rozenblum and Lieb--Thirring bounds. We will deduce it from a recent bound of Rozenblum \cite{Rozenblum22}. The difference between Rozenblum's bound and ours is that he looks at the Schr\"odinger operator with singular potential $\sigma \mathcal H^{d-1}|_{\partial\Omega}$ on all of $\R^d$, while we look at the corresponding operator on $\Omega$. Our restriction to $\Omega$ gives rise to an additional term in the bound, which, however, will be irrelevant for our purposes. An initial bound of this form was obtained in \cite{FrankLaptev08}.

\begin{proposition}\label{ltrozen}
	Let $\gamma>0$ if $d=2$ and $\gamma\geq 0$ if $d>2$. Let $\Omega\subset\R^d$ be a bounded open set with Lipschitz boundary and let $\sigma\in L^{2\gamma+d-1}(\partial\Omega)$. Then
	$$
	\Tr(-\Delta_\Omega^{(\sigma)})_\limminus^\gamma \lesssim \int_{\partial\Omega} \sigma(x)_\limminus^{2\gamma+d-1}\,d\mathcal H^{d-1}(x) + \left( \int_{\partial\Omega} \sigma(x)_\limminus\,d\mathcal H^{d-1}(x) \right)^\gamma
	$$
	with an implicit constant depending on $d$, $\gamma$ and $\Omega$.
\end{proposition}

\begin{proof}
	By the variational principle we may assume that $\sigma\leq 0$. If $P$ denotes the orthogonal projection in $L^2(\Omega)$ onto constants and if $P^\bot = 1- P$, then for $u\in H^1(\Omega)$ we find, writing $u=c+v$ with $c=Pu$ and $v=P^\bot u$,
	$$
	\int_{\partial\Omega} \sigma |u|^2\,d\mathcal H^{d-1} \geq - 2 \int_{\partial\Omega} \sigma_\limminus\,d\mathcal H^{d-1} |c|^2 - 2 \int_{\partial\Omega} \sigma_\limminus |v|^2\,d\mathcal H^{d-1} \,.
	$$
	Thus, we have
	$$
	-\Delta_\Omega^{(\sigma)} \geq - 2 \int_{\partial\Omega} \sigma_\limminus\,d\mathcal H^{d-1} P + P^\bot (-\Delta_\Omega^{(2\sigma)}) P^\bot \,,
	$$
	and consequently
	$$
	\Tr(-\Delta_\Omega^{(\sigma)})_\limminus^\gamma \leq 2^\gamma \left( \int_{\partial\Omega} \sigma_\limminus\,d\mathcal H^{d-1} \right)^\gamma + \Tr( P^\bot (-\Delta_\Omega^{(2\sigma)}) P^\bot )_\limminus^\gamma \,.
	$$
 
	Now as in \cite[Proof of Corollary 4.37]{FrankLaptevWeidl}, if $E:H^1(\Omega)\to H^1(\R^d)$ denotes the extension operator and if $v\in H^1(\Omega)$ has mean zero, then
	$$
	\int_{\R^d} |\nabla Ev|^2 \,dx \leq \|E\|_{H^1\to H^1}^2 \int_\Omega (|\nabla v|^2 + |v|^2)\,dx \leq M \int_\Omega |\nabla v|^2\,dx
	$$
	with $M= \|E\|_{H^1\to H^1}^2 (1 + P_\Omega^{-1})$, where $P_\Omega$ is the Poincar\'e constant. This implies that
	$$
	P^\bot (-\Delta_\Omega^{(2\sigma)}) P^\bot \geq P^\bot E^* (-M^{-1}\Delta + 2\sigma \mathcal H^{d-1}|_{\partial\Omega}) E P^\bot \,,
	$$
    where $-M^{-1}\Delta+2\sigma \Haus^{d-1}|_{\partial\Omega}$ denotes the Schr\"odinger operator on $\R^d$ with potential given by the singular measure $2\sigma\Haus^{d-1}|_{\partial\Omega}$ defined by means of quadratic forms; for details see~\cite{Rozenblum22}. Therefore
	\begin{align*}
		\Tr( P^\bot (-\Delta_\Omega^{(2\sigma)}) P^\bot )_\limminus^\gamma 
		& \leq \Tr (P^\bot E^* (-M^{-1}\Delta + 2\sigma \mathcal H^{d-1}|_{\partial\Omega}) E P^\bot )_\limminus^\gamma \\
		& = M^{-\gamma} \Tr (P^\bot E^* (-\Delta + 2M \sigma \mathcal H^{d-1}|_{\partial\Omega}) E P^\bot )_\limminus^\gamma \\
		& \leq M^{-\gamma} \| E P^\bot \|_{L^2\to L^2}^2 \Tr (-\Delta + 2M \sigma \mathcal H^{d-1}|_{\partial\Omega})_\limminus^\gamma \,.
	\end{align*}
    For the last inequality see, e.g., \cite[Corollary 1.31 and Lemma 1.44]{FrankLaptevWeidl}. Using the fact that $E$ is bounded from $L^2\to L^2$ (which follows from the explicit construction of the extension operator for Lipschitz domains; see, e.g., \cite[Theorem 2.92]{FrankLaptevWeidl}) and Rozenblum's bound for the Riesz means of Schr\"odinger operators with singular potentials on $\R^d$ \cite{Rozenblum22}, we obtain the claimed bound.
\end{proof}

\begin{proof}[Proof of Theorem \ref{main}]
	We shall prove the following \emph{claim}. Let $0<\gamma\leq 1$ and let $\sigma\in L^{2\gamma+d-1}(\partial\Omega)$. Then, as $\lambda\to\infty$,
	$$
	\Tr(-\Delta_\Omega^{(\sigma)}-\lambda)_\limminus^\gamma = L_{\gamma,d}^{\rm sc} |\Omega| \lambda^{\gamma+\frac d2} + \frac14 L_{\gamma,d-1}^{\rm sc} \mathcal H^{d-1}(\partial\Omega) \lambda^{\gamma+\frac{d-1}2}+ o(\lambda^{\gamma+\frac{d-1}2}) \,.
	$$
	
	Before proving this claim, let us show that it implies the assertion of the theorem. Let $\sigma\in L^p(\partial\Omega)$ for some $p>d-1$ and set $\gamma_0:=\frac12(p-d+1)>0$. Since $\partial\Omega$ has finite $\mathcal H^{d-1}$ measure, the $L^p$-spaces are nested and we have $\sigma\in L^{2\gamma+d-1}(\partial\Omega)$ for any $\gamma\leq\gamma_0$. Applying the claim, we find that the asymptotics are valid for any $\gamma\in(0,\min\{\gamma_0,1\}]$. By a simple integration argument (see \cite{AizenmanLieb} or \cite[Subsection 5.1.1]{FrankLaptevWeidl}) this implies that the asymptotics are valid for any $\gamma>0$, as claimed.
	
	To prove the claim we write, for $0<\epsilon\leq 1$,
	$$
	-\Delta_\Omega^{(\sigma)} = -(1-\epsilon)\Delta_\Omega^{(0)} - \epsilon \Delta_\Omega^{(\sigma/\epsilon)}
	$$
	and obtain, using Rotfel'd's inequality \cite[Proposition 1.43]{FrankLaptevWeidl} (here we use $\gamma\leq 1$),
	\begin{align*}
		\Tr(-\Delta_\Omega^{(\sigma)}-\lambda)_\limminus^\gamma & \leq \Tr(-(1-\epsilon)\Delta_\Omega^{(0)}-\lambda)_\limminus^\gamma + \Tr (-\epsilon\Delta_\Omega^{(\sigma/\epsilon)})_\limminus^\gamma \\
		& = \Tr(-\Delta_\Omega^{(0)}-\lambda)_\limminus^\gamma + \overline R_1(\lambda,\epsilon) + \overline R_2(\epsilon)
	\end{align*}
	with
	\begin{align*}
		\overline R_1(\lambda,\epsilon) & := \Tr(-(1-\epsilon)\Delta_\Omega^{(0)}-\lambda)_\limminus^\gamma - \Tr(-\Delta_\Omega^{(0)}-\lambda)_\limminus^\gamma \,, \\
		\overline R_2(\epsilon) & := \Tr (-\epsilon\Delta_\Omega^{(\sigma/\epsilon)})_\limminus^\gamma \,.
	\end{align*}
	Similarly, writing
	$$
	-(1+\epsilon)\Delta_\Omega^{(0)} = -\Delta_\Omega^{(\sigma)} - \epsilon \Delta_\Omega^{(-\sigma/\epsilon)} \,,
	$$
	we find
	\begin{align*}
		\Tr(-\Delta_\Omega^{(\sigma)}-\lambda)_\limminus^\gamma & \geq \Tr(-(1+\epsilon)\Delta_\Omega^{(0)}-\lambda)_\limminus^\gamma - \Tr (-\epsilon\Delta_\Omega^{(\sigma/\epsilon)})_\limminus^\gamma \\
		& = \Tr(-\Delta_\Omega^{(0)}-\lambda)_\limminus^\gamma - \underline R_1(\lambda,\epsilon) - \underline R_2(\epsilon)
	\end{align*}
	with
	\begin{align*}
		\underline R_1(\lambda,\epsilon) & := \Tr(-\Delta_\Omega^{(0)}-\lambda)_\limminus^\gamma -  \Tr(-(1+\epsilon)\Delta_\Omega^{(0)}-\lambda)_\limminus^\gamma \,, \\
		\underline R_2(\epsilon) & := \Tr (-\epsilon\Delta_\Omega^{(-\sigma/\epsilon)})_\limminus^\gamma \,.
	\end{align*}
	
	As a direct consequence of Proposition \ref{ltrozen}, we have
	$$
	\overline R_2(\epsilon) \lesssim \epsilon^{-\gamma-d+1}
	\qquad\text{and}\qquad
	\underline R_2(\epsilon) \lesssim \epsilon^{-\gamma-d+1} \,.
	$$
	The implied constants here depend on $\Omega$ and $\sigma$, as well as $\gamma$ and $d$. Moreover, note that the second term in the inequality in Proposition \ref{ltrozen} diverges slower as $\epsilon\to 0$ than the first one.
	
	From the two-term Weyl asymptotics for Riesz means for the Neumann Laplacian proved in \cite[Theorem 1.1]{FrankLarson_24} it follows that
	$$
	|\overline R_1(\lambda,\epsilon)| \lesssim \epsilon \lambda^{\gamma+\frac d2} + o(\lambda^{\gamma+\frac{d-1}{2}})
	\qquad\text{and}\qquad
	|\underline R_1(\lambda,\epsilon)| \lesssim \epsilon \lambda^{\gamma+\frac d2} + o(\lambda^{\gamma+\frac{d-1}{2}}) \,.
	$$
	The $\epsilon \lambda^{\gamma+ \frac{d}2}$ terms here come from the difference in the main terms, while the $o$-terms come from the remainders. Note that we use a cancelation of the volume and the boundary terms.
	
	To summarize, we have shown that
	$$
	|\Tr(-\Delta_\Omega^{(\sigma)}-\lambda)_\limminus^\gamma - \Tr(-\Delta_\Omega^{(0)}-\lambda)_\limminus^\gamma|
	\lesssim \epsilon \lambda^{\gamma+\frac d2} +  \epsilon^{-\gamma-d+1} + o(\lambda^{\gamma+\frac{d-1}{2}})\,.
	$$
	
	We now choose $\epsilon$ to minimize the error terms, that is, $\epsilon \lambda^{\gamma+\frac d2} =  \epsilon^{-\gamma-d+1}$, which is the same as $\epsilon=\lambda^{-\frac{\gamma+\frac d2}{\gamma+d}}$. The crucial point is that $\epsilon = o(\lambda^{-\frac12})$, which is readily verified, using $\gamma>0$. As a consequence we have
	$$
	\epsilon \lambda^{\gamma+\frac d2} + \epsilon^{-\gamma-d+1} = o(\lambda^{-\frac12}) \lambda^{\gamma+\frac d2} = o(\lambda^{\gamma+\frac{d-1}{2}}) \,.
	$$
	Thus, we obtain
	$$
	\Tr(-\Delta_\Omega^{(\sigma)}-\lambda)_\limminus^\gamma - \Tr(-\Delta_\Omega^{(0)}-\lambda)_\limminus^\gamma
	= o(\lambda^{\gamma+\frac{d-1}{2}})\,.
	$$
	Using once again the two-term asymptotics for Riesz means of the Neumann Laplacian in~\cite[Theorem 1.1]{FrankLarson_24}, we arrive at the claimed asymptotics.
\end{proof}

We find it remarkable that if instead of Rotfel'd's inequality we use the simpler inequality in \cite[Proposition 1.40]{FrankLaptevWeidl}, we are not able to make the proof work. It only gives an order-sharp remainder bound. Also, one can use \cite[Proposition 1.42]{FrankLaptevWeidl} to prove asymptotics for $\gamma>1$, but this does not improve the final result, since the main work consists in obtaining asymptotics under minimal $L^p$ assumptions. Under stronger integrability assumptions on $\sigma$ and regularity assumptions on $\Omega$ one can improve the $o$-remainder term in Theorem \ref{main}, but we have not investigated this.


\section{The Robin heat kernel}

In the next section we will prove Theorem \ref{thm: Difference of Riesz means intro}. To do so, we will use a technique from~\cite{FrankLarson_24} that allows us to deduce the Riesz means asymptotics from heat trace asymptotics. The present section is a preparation of the proof of these heat trace asymptotics.

Throughout this section we assume that $\sigma\in L^p(\partial\Omega)$ for some $p>1$ if $d=2$ and $p\geq d-1$ if $d>2$. Additional assumptions will be imposed later on.


\subsection{Gaussian upper bound}

Let $k_\Omega^{(\sigma)}\colon (0, \infty)\times \Omega \times \Omega \to \R$ be the integral kernel of the semigroup $e^{t\Delta_\Omega^{(\sigma)}}$, that is,
\begin{equation*}
 	(e^{t\Delta_\Omega^{(\sigma)}}f)(x) = \int_\Omega k_\Omega^{(\sigma)}(t, x, y)f(y)\,dy\,.
\end{equation*} 
It is well known (and follows from the Beurling--Deny criterion \cite[Section 1.3]{Davies_heatkernels} -- note that we apply this to the operator $-\Delta_\Omega^{(\sigma)}-\lambda_1(-\Delta_\Omega^{(\sigma)})\geq 0$) that $e^{t\Delta_\Omega^{(\sigma)}}$ is positivity preserving and therefore $k_\Omega^{(\sigma)}\geq 0$.

Our results in this section are conditional on the fact that the heat kernel satisfies a Gaussian upper bound. More precisely, we work under the following assumption.

\begin{assumption}\label{gaussianboundass}
    There are $M_\sigma, C, t_0>0, \Lambda\geq 0$ such that
    \begin{equation}\label{eq: Assumed Gaussian bound}
    	k_\Omega^{(\sigma)}(t, x, x') \leq M_\sigma \max\{1, (t/t_0)^{\frac{d}2}\}t^{-\frac{d}2}e^{-\frac{|x-x'|^2}{Ct}+\Lambda t} \quad \mbox{for all }t>0, x, x'\in \Omega \,.
    \end{equation}
\end{assumption}

The following proposition gives a sufficient condition for this assumption to hold.

\begin{proposition}\label{gaussianbound}
    If $\sigma_\limminus\in L^\infty(\partial\Omega)$, then Assumption \ref{gaussianboundass} is satisfied. 
\end{proposition}

\begin{remark}
    Proposition \ref{gaussianbound} remains valid when our standing assumption on $\sigma_\limplus$ (namely $\sigma_\limplus\in L^p(\partial\Omega)$ with some $p>1$ for $d=2$ and $p\geq d-1$ for $d>2$) is replaced by the weak assumption $\sigma_\limplus\in L^1(\partial\Omega)$. In this case the operator is defined with form domain $\{u\in H^1(\Omega):\ \sigma_\limplus |u|_{\partial\Omega}|^2 \in L^1(\partial\Omega)\}$. Since for our results in this section we will need to impose stronger integrability assumptions on $\sigma_\limplus$ we omit a proof of this assertion. 
\end{remark}

\begin{proof}
    Under the assumption $\sigma\in L^\infty(\partial\Omega)$, the assertion is shown in \cite{ElstWong_20}, extending the previous result in \cite{Daners_00} for nonnegative, bounded $\sigma$. (The paper \cite{ElstWong_20} also mentions \cite{Daners_09}, where we cannot find the claimed result, however.) In passing we mention that the assertion is also known under the assumption $\sigma\geq 0$, see \cite{Gesztesy_etal_PAMS15}, but we will not need this result but rather reprove it.

    The fact that Assumption \ref{gaussianboundass} remains valid also for unbounded $\sigma_\limplus$ follows from the pointwise bound
    $$
    k_\Omega^{(\sigma)}(t, x, x') \leq k_\Omega^{(-\sigma_\limminus)}(t, x, x')
    \qquad\text{for all}\ x,x'\in\Omega \,.
    $$
    To justify this bound, we use the Duhamel formula
    $$
    \bigl( f, e^{t\Delta^{(\sigma)}_\Omega} g \bigr) = \bigl( f, e^{t\Delta^{(-\sigma_\limminus)}_\Omega} g \bigr) - \int_0^t \int_{\partial\Omega} \bigl(e^{(t-s)\Delta_\Omega^{(-\sigma_\limminus)}}f \bigr)(y) \bigl( e^{s\Delta_\Omega^{(\sigma)}}g\bigr)(y) \sigma(y)_\limplus \,d\mathcal H^{d-1}(y)\,ds \,,
    $$
    valid for all $f,g\in L^2(\Omega)$. Note that the boundary integral is well defined since $e^{(t-s)\Delta_\Omega^{(-\sigma_\limminus)}}f$ and $e^{s\Delta_\Omega^{(\sigma)}}g$ belong to $H^1(\Omega)$, which is the common form domain of $-\Delta^{(-\sigma_\limminus)}_\Omega$ and $-\Delta_\Omega^{(\sigma)}$. Since $e^{(t-s)\Delta_\Omega^{(-\sigma_\limminus)}}$ and $e^{s\Delta_\Omega^{(\sigma)}}$ are positivity preserving, we see that for $f,g\geq 0$, we have
    $$
    \int_{\partial\Omega} \bigl(e^{(t-s)\Delta_\Omega^{(-\sigma_\limminus)}}f \bigr)(y) \bigl( e^{s\Delta_\Omega^{(\sigma)}}g\bigr)(y) \sigma(y)_\limplus \,d\mathcal H^{d-1}(y) \geq 0 \,.
    $$
    Thus, for $f,g\geq 0$ we have $( f, e^{t\Delta^{(\sigma)}_\Omega} g ) \leq ( f, e^{t\Delta^{(-\sigma_\limminus)}_\Omega} g )$, which implies the claimed pointwise bound.    
\end{proof}


\subsection{Approximation of the Robin heat kernel}

We shall construct a sequence of approximations $K_j$ of $k_\Omega^{(\sigma)}$ in terms of the Neumann heat kernel $k_\Omega^{(0)}$. The approximations are defined recursively by
$$
K_0(t, x, x'):=0
$$ 
and, for $j\geq 1$,
\begin{equation*}
	K_j(t, x, x') := k_\Omega^{(0)}(t, x, x') - \int_0^t \int_{\partial\Omega} k_\Omega^{(0)}(t-s, x, y)K_{j-1}(s, y, x')\sigma(y)\,d\Haus^{d-1}(y)ds\,.
\end{equation*}
Explicitly, we have
\begin{align*}
	K_0(t, x, x')=0\,, \quad K_1(t, x, x') = k_\Omega^{(0)}(t, x, x')\,,
\end{align*}
and, for $j\geq 2$,
\begin{align*}
	K_j(t, x, x') &= k_\Omega^{(0)}(t, x, x')\\
	&\quad +\sum_{i=1}^{j-1}(-1)^i \int_{\{s \in (0, t)^i: s_l>s_{l+1}\}}\int_{(\partial\Omega)^{i}}\Bigl(\prod_{l=0}^{i}k_\Omega^{(0)}(s_l-s_{l+1}, y_l, y_{l+1})\Bigr)\Bigl|_{\substack{s_0=t,\hspace{11pt}\\ s_{i+1}=0, \hspace{2pt}\\ y_0=x,\hspace{9pt} \\ y_{i+1}=x'}} \\
	&\quad \quad \times\Bigl(\prod_{l=1}^{i}\sigma(y_i)\Bigr)\Bigl(\prod_{l=1}^id\Haus^{d-1}(y_{l})\Bigr)\Bigl(\prod_{l=1}^ids_{l}\Bigr)\,.
\end{align*}

\begin{proposition}\label{prop: improved kernel approx}
Assume that $\Omega \subset \R^d$ is open, bounded, with Lipschitz regular boundary, that $\sigma \in L^p(\partial\Omega)$ with $p>d-1$, and that Assumption \ref{gaussianboundass} is satisfied. Then, for each $j\geq 0$,
\begin{align*}
	|k_\Omega^{(\sigma)}(&t, x, x') - K_j(t, x, x')| \\
	&\lesssim_{\Omega, j,p} M_\sigma  \|\sigma\|_{L^p(\partial\Omega)}^j\max\{1, (t/t_0)^{\frac{(j+1)d}2}\}t^{-\frac{d-j}{2}- j\frac{d-1}{2p}} e^{-\frac{|x-x'|^2+(d_\Omega(x')+d_\Omega(x))^2}{2^jC t}+ \Lambda t}\,.
\end{align*}
\end{proposition}

By assumption $p>d-1$, for any fixed $\alpha$ we can therefore choose $j$ so large that the error in the approximation is $o(t^\alpha)$ as $t \to 0^\limplus$. However, as the constructed approximation $K_j$ involves iterated integrals of $k_\Omega^{(0)}$ and $\sigma$, it is not an easy task to precisely understand the size of the correction coming from these integrals. Heuristically, the integral term with index $i$ should be $O(t^{-\frac{d-i}2})$ when $|x-x'|\lesssim \sqrt{t}, d_\Omega(x) \lesssim \sqrt{t}$ and exponentially small everywhere else.

Our proof of Proposition~\ref{prop: improved kernel approx} uses the following lemma.
\begin{lemma}\label{lem: Gaussian bdry integral bound}
    Assume that $\Omega \subset \R^d$ is open, bounded, with Lipschitz regular boundary, that $\sigma \in L^p(\partial\Omega)$ for some $p>1$. Fix $a, b < \frac{d+1}2-\frac{d-1}{2p}$. Then for all $x, x'\in \Omega, t>0, C>0,$
    \begin{align*}
		 \int_0^{t}\int_{\partial\Omega} &|\sigma(y)|\frac{e^{- \frac{|x-y|^2}{C(t-s)}-\frac{|x'-y|^2}{Cs}}}{(t-s)^{a}s^{b}}\,d\Haus^{d-1}(y)ds\\
   &\lesssim_{d,p,a,b} C^{\frac{d-1}{2p'}} \mathcal{M}(\Omega)^{\frac{1}{p'}} \|\sigma\|_{L^{p}(\partial\Omega)} t^{-(a+b)+\frac{d+1}2-\frac{d-1}{2p}}e^{-\frac{(d_\Omega(x)+d_\Omega(x'))^2}{2Ct}- \frac{|x-x'|^2}{2Ct}}\,,
	\end{align*}
    where 
    \begin{equation*}
        \mathcal{M}(\Omega):= \sup_{r>0, z\in \R^d}\frac{\Haus^{d-1}(\partial\Omega \cap B_r(z))}{r^{d-1}}\,.
    \end{equation*}
\end{lemma}

Note that, since $\Omega$ is bounded and has Lipschitz regular boundary, the supremum defining $\mathcal{M}(\Omega)$ is finite.

\begin{proof}
    Since $\frac{|x-y|^2}{t-s}+\frac{|x'-y|^2}{s}\geq \frac{(d_\Omega(x)+d_\Omega(x'))^2}{t}$ for all $s\in (0, t)$ and $y \in \partial\Omega$, it follows that
	\begin{equation}\label{eq: extracting decay}
		e^{- \frac{|x-y|^2}{C(t-s)}-\frac{|x'-y|^2}{Cs}} \leq e^{- \frac{(d_\Omega(x)+d_\Omega(x'))^2}{2Ct}} e^{- \frac{|x-y|^2}{2C(t-s)}-\frac{|x'-y|^2}{2Cs}}\,.
	\end{equation}
	Therefore,
	\begin{equation}
	\begin{aligned}
		 \int_0^{t}&\int_{\partial\Omega} |\sigma(y)|\frac{e^{- \frac{|x-y|^2}{C(t-s)}-\frac{|x'-y|^2}{Cs}}}{(t-s)^{a}s^{b}}\,d\Haus^{d-1}(y)ds\\
			&\leq 
				e^{- \frac{(d_\Omega(x)+d_\Omega(x'))^2}{2Ct}}\int_0^{t}\int_{\partial\Omega} |\sigma(y)|\frac{e^{- \frac{1}{2C}\bigl(\frac{|x-y|^2}{t-s}+\frac{|x'-y|^2}{s}\bigr)}}{(t-s)^{a}s^{b}}\,d\Haus^{d-1}(y)ds\\
			&= 
				t^{-(a+b)+1}e^{- \frac{(d_\Omega(x)+d_\Omega(x'))^2}{2Ct}}\int_0^1\int_{\partial\Omega}|\sigma(y)| \frac{e^{- \frac{1}{2Ct}\bigl(\frac{|x-y|^2}{1-\tau}+\frac{|x'-y|^2}{\tau}\bigr)}}{(1-\tau)^{a}\tau^{b}}\,d\Haus^{d-1}(y)d\tau\\
				&= 
				t^{-(a+b)+1}e^{-\frac{(d_\Omega(x)+d_\Omega(x'))^2}{2Ct}- \frac{|x-x'|^2}{2Ct}}\int_0^1\int_{\partial\Omega}|\sigma(y)| \frac{e^{- \frac{|y-(\tau x+(1-\tau)x')|^2}{2Ct(1-\tau)\tau}}}{(1-\tau)^{a}\tau^{b}}\,d\Haus^{d-1}(y)d\tau\,. \hspace{-25pt}
	\end{aligned}
	\end{equation}

    By H\"older's inequality
    \begin{align*}
    \int_0^1\int_{\partial\Omega}&|\sigma(y)| \frac{e^{- \frac{|y-(\tau x+(1-\tau)x')|^2}{2Ct(1-\tau)\tau}}}{(1-\tau)^{a}\tau^{b}}\,d\Haus^{d-1}(y)d\tau\\
    &=
     \int_0^1\frac{1}{(1-\tau)^{a}\tau^{b}}\biggl(\int_{\partial\Omega}|\sigma(y)|e^{- \frac{|y-(\tau x+(1-\tau)x')|^2}{2Ct(1-\tau)\tau}} \,d\Haus^{d-1}(y)\biggr)d\tau\\
     &\leq
     \|\sigma\|_{L^{p}(\partial\Omega)}\int_0^1\frac{1}{(1-\tau)^{a}\tau^{b}}\biggl(\int_{\partial\Omega}e^{- p'\frac{|y-(\tau x+(1-\tau)x')|^2}{2Ct(1-\tau)\tau}} \,d\Haus^{d-1}(y)\biggr)^{\frac{1}{p'}}d\tau\,.
    \end{align*}
	Using the co-area formula we can estimate the inner integral as
	\begin{align*}
		\int_{\partial\Omega} &e^{- p'\frac{|y-(\tau x+(1-\tau)x')|^2}{2Ct(1-\tau)\tau}}\,d\Haus^{d-1}(y)\\
		&=\int_{0}^1 \Haus^{d-1}(\partial\Omega \cap \{y\in \R^d: e^{- p'\frac{ |y-(\tau x+(1-\tau)x')|^2}{2Ct(1-\tau)\tau}}\geq \eta\})\,d\eta\\
		&=
		\int_0^1\Haus^{d-1}(\partial\Omega \cap B_{\sqrt{2Ct(1-\tau)\tau \ln(1/\eta)/p'}}(\tau x + (1-\tau)x'))\,d\eta\\
		&\leq
		\Bigl(\frac{2C}{p'}\Bigr)^{\frac{d-1}2}\mathcal{M}(\Omega)t^{\frac{d-1}{2}}(1-\tau)^{\frac{d-1}{2}}\tau^{\frac{d-1}{2}} \int_0^1  \ln(1/\eta)^{\frac{d-1}{2}}\,d\eta\\
		&=
		\Gamma\Bigl(\frac{d+1}2\Bigr)\Bigl(\frac{2C}{p'}\Bigr)^{\frac{d-1}2}\mathcal{M}(\Omega)t^{\frac{d-1}{2}}(1-\tau)^{\frac{d-1}{2}}\tau^{\frac{d-1}{2}}\,.
	\end{align*}

	Hence
	\begin{align*}
		\int_0^1&\frac{1}{(1-\tau)^{a}\tau^{b}}\biggl(\int_{\partial\Omega} e^{- p'\frac{|y-(\tau x+(1-\tau)x')|^2}{2Ct(1-\tau)\tau}}\,d\Haus^{d-1}(y)\biggr)^{\frac{1}{p'}}d\tau\\
		&\leq
		\Gamma\Bigl(\frac{d+1}2\Bigr)^{\frac{1}{p'}}\Bigl(\frac{2C}{p'}\Bigr)^{\frac{d-1}{2p'}}\mathcal{M}(\Omega)^{\frac{1}{p'}}t^{\frac{d-1}{2p'}}\int_0^1\frac{1}{(1-\tau)^{a-\frac{d-1}{2p'}}\tau^{b-\frac{d-1}{2p'}}}\,d\tau\\
		&=
		\Gamma\Bigl(\frac{d+1}2\Bigr)^{\frac{1}{p'}}\Bigl(\frac{2C}{p'}\Bigr)^{\frac{d-1}{2p'}}\mathcal{M}(\Omega)^{\frac{1}{p'}}t^{\frac{d-1}{2p'}}\frac{\Gamma(1-a+\frac{d-1}{2p'})\Gamma(1-b+\frac{d-1}{2p'})}{\Gamma(2-a-b+\frac{d-1}{p'})}\,,
	\end{align*}
	where we used that $\max\{a,b\}- \frac{d-1}{2p'}<1$ since $a, b < \frac{d+1}2-\frac{d-1}{2p}$.

    Putting the above together we have proved that
    \begin{align*}
    	 \int_0^{t}&\int_{\partial\Omega} |\sigma(y)|\frac{e^{- \frac{|x-y|^2}{C(t-s)}-\frac{|x'-y|^2}{Cs}}}{(t-s)^{a}s^{b}}\,d\Haus^{d-1}(y)ds\\
    	 &\lesssim_{d,p,a,b} C^{\frac{d-1}{2p'}}\mathcal{M}(\Omega)^{\frac{1}{p'}} \|\sigma\|_{L^{p}(\partial\Omega)} t^{-(a+b)+1+\frac{d-1}{2p'}}e^{-\frac{(d_\Omega(x)+d_\Omega(x'))^2}{2Ct}- \frac{|x-x'|^2}{2Ct}}\\
    	 &=C ^{\frac{d-1}{2p'}}\mathcal{M}(\Omega)^{\frac{1}{p'}} \|\sigma\|_{L^{p}(\partial\Omega)} t^{-(a+b)+\frac{d+1}2-\frac{d-1}{2p}}e^{-\frac{(d_\Omega(x)+d_\Omega(x'))^2}{2Ct}- \frac{|x-x'|^2}{2Ct}}\,,
    \end{align*}
    which completes the proof of Lemma~\ref{lem: Gaussian bdry integral bound}.
\end{proof}

\begin{proof}[Proof of Proposition~\ref{prop: improved kernel approx}]
	We argue by induction in $j$. For $j=0$ we know from \eqref{eq: Assumed Gaussian bound} that
	\begin{equation*}
		|k_\Omega^{(\sigma)}(t, x, x')-K_0(t, x, x')| = |k_\Omega^{(\sigma)}(t, x, x')|\leq M_\sigma \max\{1, (t/t_0)^{\frac{d}2}\}t^{-\frac{d}2}e^{-\frac{|x-x'|^2}{Ct}+\Lambda t}\,,
	\end{equation*}
	which is the claimed bound.

	Fix $j\geq 1$ and assume that the bound in the proposition holds for $|k^{(\sigma)}-K_{j-1}|$. We define 
	\begin{align*}
		w(t, x, x') &:= k_\Omega^{(\sigma)}(t, x, x') - k_\Omega^{(0)}(t, x, x')\,,\\
		W(t, x, x') &:= k_\Omega^{(\sigma)}(t, x, x') - K_{j-1}(t, x, x')\,.
	\end{align*}
	Note that for any fixed $x'\in \Omega$ the function $(t, x) \mapsto w(t, x, x')$ solves
	\begin{equation*}
		\begin{cases}
			(\partial_t-\Delta_{x})w(t, x, x') = 0 & \mbox{for }(t, x)\in (0, \infty)\times \Omega\,,\\
			w(0, x, x') =0 &\mbox{for }x\in \Omega\,,\\
			\frac{\partial}{\partial \nu_{x}}w(t, x, x') = -\sigma(x)k_\Omega^{(\sigma)}(t, x, x') & \mbox{for }(x, t)\in \partial\Omega\,,
		\end{cases}
	\end{equation*}
    in the weak sense, where $\frac{\partial}{\partial \nu_x}$ denotes the normal derivative to $\partial\Omega$ in the outward pointing direction in the $x$-variable. 

	By Duhamel's principle,
	\begin{align*}
		k^{(\sigma)}(t, x, x')-k^{(0)}_\Omega(t, x, x') &= w(t, x, x') \\
        &= -\int_0^t\int_{\partial\Omega}k_\Omega^{(0)}(t-s, x, y)\sigma(y)k_\Omega^{(\sigma)}(s, y, x')\,d\Haus^{d-1}(y)ds\\
        &= -\int_0^t\int_{\partial\Omega}k_\Omega^{(0)}(t-s, x, y)\sigma(y)K_{j-1}(s, y, x')\,d\Haus^{d-1}(y)ds\\
        &\quad -\int_0^t\int_{\partial\Omega}k_\Omega^{(0)}(t-s, x, y)\sigma(y)W(s, y, x')\,d\Haus^{d-1}(y)ds\,.
	\end{align*}
    Rearranging and using the recursive definition of $K_j$ we arrive at
    \begin{equation*}
        k_\Omega^{(\sigma)}(t, x, x')-K_j(t, x, x') = -\int_0^t\int_{\partial\Omega}k_\Omega^{(0)}(t-s, x, y)\sigma(y)W(s, y, x')\,d\Haus^{d-1}(y)ds\,.
    \end{equation*}

	By the assumed bound for $W$ and the fact that $k^{(0)}_\Omega$ satisfies
	\begin{equation*}
		|k_\Omega^{(0)}(t, x, x')| \leq M_0 \max\{1, (t/t_0)^{\frac{d}2}\}t^{-\frac{d}2}e^{-\frac{|x-x'|^2}{Ct}} \quad \mbox{for all }t>0, x, x'\in \Omega\,,
	\end{equation*}
	we find
	\begin{align*}
		\biggl|\int_0^t\int_{\partial\Omega}& k_\Omega^{(0)}(t-s, x, y)\sigma(y)W(s, y, x')\,d\Haus^{d-1}(y)ds\Biggr| \\
        &\lesssim_{\Omega, j, p} M_\sigma\|\sigma\|_{L^p(\partial\Omega)}^{j-1} e^{\Lambda t}\int_0^t\int_{\partial\Omega}\frac{|\sigma(y)|\max\{1, ((t-s)/t_0)^{\frac{d}2}\}\max\{1, (s/t_0)^{\frac{dj}2}\}}{(t-s)^{\frac{d}2}s^{\frac{d-j+1}{2}+(j-1)\frac{d-1}{2p}}}\\
        &\quad \quad \times e^{- \frac{|y-x'|^2}{2^{j-1}C s}- \frac{|x-y|^2}{2^{j-1}C(t-s)}}\,d\Haus^{d-1}(y)ds\,.
	\end{align*}

	If $t\geq 2t_0$ then
	\begin{align*}
		\int_0^t&\int_{\partial\Omega}\frac{|\sigma(y)|\max\{1, ((t-s)/t_0)^{\frac{d}2}\}\max\{1, (s/t_0)^{\frac{dj}2}\}}{(t-s)^{\frac{d}2}s^{\frac{d-j+1}{2}+(j-1)\frac{d-1}{2p}}}e^{- \frac{|y-x'|^2}{2^{j-1}C s}- \frac{|x-y|^2}{2^{j-1}C(t-s)}}\,d\Haus^{d-1}(y)ds\\
		&=
		t_0^{-\frac{d}2}\int_0^{t_0}\int_{\partial\Omega}\frac{|\sigma(y)|}{s^{\frac{d-j+1}{2}+(j-1)\frac{d-1}{2p}}}e^{- \frac{|y-x'|^2}{2^{j-1}C s}- \frac{|x-y|^2}{2^{j-1}C(t-s)}}\,d\Haus^{d-1}(y)ds\\
		&\quad +
		t_0^{-\frac{(j+1)d}2}\int_{t_0}^{t-t_0}\int_{\partial\Omega}|\sigma(y)|s^{(j-1)(\frac{d+1}{2}-\frac{d-1}{2p})}e^{- \frac{|y-x'|^2}{2^{j-1}C s}- \frac{|x-y|^2}{2^{j-1}C(t-s)}}\,d\Haus^{d-1}(y)ds\\
		&\quad +
		t_0^{-\frac{jd}2}\int_{t-t_0}^t\int_{\partial\Omega}\frac{|\sigma(y)|}{(t-s)^{\frac{d}2}}s^{(j-1)(\frac{d+1}{2}-\frac{d-1}{2p})}e^{- \frac{|y-x'|^2}{2^{j-1}C s}- \frac{|x-y|^2}{2^{j-1}C(t-s)}}\,d\Haus^{d-1}(y)ds
	\end{align*}
    If instead $t<2t_0$ then
    \begin{align*}
		\int_0^t&\int_{\partial\Omega}\frac{|\sigma(y)|\max\{1, ((t-s)/t_0)^{\frac{d}2}\}\max\{1, (s/t_0)^{\frac{dj}2}\}}{(t-s)^{\frac{d}2}s^{\frac{d-j+1}{2}+(j-1)\frac{d-1}{2p}}}e^{- \frac{|y-x'|^2}{2^{j-1}C s}- \frac{|x-y|^2}{2^{j-1}C(t-s)}}\,d\Haus^{d-1}(y)ds\\
		&\leq
		2^{\frac{(j+1)d}2}\int_0^t\int_{\partial\Omega}\frac{|\sigma(y)|}{(t-s)^{\frac{d}2}s^{\frac{d-j+1}{2}+(j-1)\frac{d-1}{2p}}}e^{- \frac{|y-x'|^2}{2^{j-1}C s}- \frac{|x-y|^2}{2^{j-1}C(t-s)}}\,d\Haus^{d-1}(y)ds\,.
	\end{align*}
	Consequently, we conclude that
	\begin{equation}\label{eq: bound1}
	\begin{aligned}
		|k^{(\sigma)}&(t, x, x')-K_j(t, x, x')| \\
        &\lesssim_{\Omega, j, p} M_\sigma \|\sigma\|_{L^p(\partial\Omega)}^{j-1} e^{\Lambda t}\\
        &\quad \times\max\Biggl\{
		t_0^{-\frac{d}2}\int_0^{t}\int_{\partial\Omega}\frac{|\sigma(y)|}{s^{\frac{d-j+1}{2}+(j-1)\frac{d-1}{2p}}}e^{- \frac{|y-x'|^2}{2^{j-1}C s}- \frac{|x-y|^2}{2^{j-1}C(t-s)}}\,d\Haus^{d-1}(y)ds\,,\\
		&\qquad
		t_0^{-\frac{(j+1)d}2}\int_{0}^{t}\int_{\partial\Omega}|\sigma(y)|s^{(j-1)(\frac{d+1}{2}-\frac{d-1}{2p})}e^{- \frac{|y-x'|^2}{2^{j-1}C s}- \frac{|x-y|^2}{2^{j-1}C(t-s)}}\,d\Haus^{d-1}(y)ds\,,\\
		&\qquad 
		t_0^{-\frac{jd}2}\int_{0}^t\int_{\partial\Omega}\frac{|\sigma(y)|}{(t-s)^{\frac{d}2}}s^{(j-1)(\frac{d+1}{2}-\frac{d-1}{2p})}e^{- \frac{|y-x'|^2}{2^{j-1}C s}- \frac{|x-y|^2}{2^{j-1}C(t-s)}}\,d\Haus^{d-1}(y)ds\,,\\
		&\qquad +\int_0^t\int_{\partial\Omega}\frac{|\sigma(y)|}{(t-s)^{\frac{d}2}s^{\frac{d-j+1}{2}+(j-1)\frac{d-1}{2p}}}e^{- \frac{|y-x'|^2}{2^{j-1}C s}- \frac{|x-y|^2}{2^{j-1}C(t-s)}}\,d\Haus^{d-1}(y)ds
        \Biggr\}\,.
	\end{aligned}
	\end{equation}
	To complete the proof we need to find appropriate estimates for integrals of the form
	\begin{equation}
		 \int_0^{t}\int_{\partial\Omega} |\sigma(y)|\frac{e^{- \frac{|x-y|^2}{2^{j-1}C(t-s)}-\frac{|x'-y|^2}{2^{j-1}Cs}}}{(t-s)^{a}s^{b}}\,d\Haus^{d-1}(y)ds\,.
	\end{equation}
	for the particular values of $a, b$ that appear in \eqref{eq: bound1}. Note that for each of the integrals in \eqref{eq: bound1} the powers $a, b$ satisfy $a, b \leq \frac{d}2 < \frac{d+1}2-\frac{d-1}{2p}$ as $p>d-1$. Therefore, we can bound each of these integrals by using Lemma~\ref{lem: Gaussian bdry integral bound}. This completes the proof of Proposition~\ref{prop: improved kernel approx}.
\end{proof}


\section{Asymptotics for the difference of traces}
\label{sec: Proofs asymptotics of differences}

\subsection{Differences of heat traces}

Let us explore some consequences of Proposition~\ref{prop: improved kernel approx} for the heat traces. As
\begin{align*}
	\Tr(e^{t\Delta_\Omega^{(\sigma)}}) =
	\int_\Omega k_\Omega^{(\sigma)}(t, x, x)\,dx\,,
\end{align*}
one deduces that, for any $j\geq 0$,
\begin{align*}
	\Tr(e^{t\Delta_\Omega^{(\sigma)}})- \Tr(e^{t\Delta_\Omega^{(0)}})
	&=
	\int_\Omega k_\Omega^{(\sigma)}(t, x, x)\,dx -\int_\Omega k_\Omega^{(0)}(t, x, x)\,dx\\
	&=
	\int_\Omega (k_\Omega^{(\sigma)}(t, x, x)-K_j(t, x, x))\,dx\\
	&\quad +\int_\Omega(K_j(t, x, x)-k_\Omega^{(0)}(t, x, x))\,dx\,.
\end{align*}
Under the assumptions of from Proposition~\ref{prop: improved kernel approx} we have
\begin{align*}
 	\biggl|\int_\Omega (k_\Omega^{(\sigma)}&(t, x, x)-K_j(t, x, x))\,dx\biggr| \lesssim \max\{1, (t/t_0)^{\frac{(j+1)d}2}\}t^{-\frac{d-j}{2}- j\frac{d-1}{2p}}e^{\Lambda t} \int_\Omega e^{-\frac{d_\Omega(x)^2}{2^jC t}}\,dx\,.
\end{align*} 
As $\Omega$ is bounded with Lipschitz boundary, we have $|\{x\in \Omega: d_\Omega(x)<\eta\}| = \eta\Haus^{d-1}(\partial\Omega)(1+o_{\eta \to 0^\limplus}(1))$. Therefore, by using the layercake formula, one finds
\begin{equation*}
	\biggl|\int_\Omega e^{-\frac{d_\Omega(x)^2}{2^jC t}}\,dx \biggr| \lesssim \sqrt{t} \quad \mbox{as }t\to 0^\limplus\,.
\end{equation*}

Later, in the proof of Theorem \ref{thm: Difference of Riesz means intro} we will choose $j=2$, but before doing so, let us discuss the results that we obtain with the choice $j=1$. By combining the above bounds, we find
\begin{equation}\label{eq: trace estimate j=1}
	\Tr(e^{t\Delta_\Omega^{(\sigma)}})- \Tr(e^{t\Delta_\Omega^{(0)}})
	=
	 O(t^{-\frac{d-2}{2}- \frac{d-1}{2p}})\,.
\end{equation}

We note that since $p>d-1$ we have $\frac{d-2}2+\frac{d-1}{2p}< \frac{d-1}{2}$. Consequently, the estimate~\eqref{eq: trace estimate j=1} and the two-term asymptotics for $\Tr(e^{t\Delta^{(0)}_\Omega})$ proved in~\cite{Brown93} imply that
\begin{equation}
    \label{eq:heatasymp}
	\Tr(e^{t\Delta_\Omega^{(\sigma)}}) = (4\pi t)^{-\frac{d}2}\Bigl(|\Omega| + \frac{\sqrt{\pi t}}2 \Haus^{d-1}(\partial\Omega)+ o(\sqrt{t})\Bigr) \quad \mbox{as }t \to 0^\limplus\,.
\end{equation}
By using the strategy used to prove \cite[Theorem 1.1]{FrankLarson_24} with Neumann boundary conditions, this two-term heat trace asymptotics can actually be used to give a second proof of Theorem~\ref{main} under the additional assumption that the Gaussian upper bound~\eqref{eq: Assumed Gaussian bound} holds for $k_\Omega^{(\sigma)}$.

In general, one cannot expect to have any further terms in the asymptotic expansion of $\Tr(e^{t\Delta^{(\sigma)}})$ than in \eqref{eq:heatasymp} when $\partial\Omega$ is merely Lipschitz. Indeed, according to Brown~\cite{Brown93} the error term $o(\sqrt{t})$ in the asymptotic expansion for the Neumann problem is sharp on the algebraic scale, i.e.\ for any $\alpha>1/2$ there exists a bounded Lipschitz sets $\Omega$ so that
\begin{equation*}
	\limsup_{t\to 0^\limplus} \frac{\bigl|(4\pi t)^{\frac{d}2}\Tr(e^{t\Delta_\Omega^{(0)}})- |\Omega| -\frac{\sqrt{\pi t}}2 \Haus^{d-1}(\partial\Omega)\bigr|}{t^{\alpha}} = \infty\,.
\end{equation*}
By~\eqref{eq: trace estimate j=1}, the same conclusion holds for the Robin Laplace operators satisfying the assumptions of Proposition~\ref{prop: improved kernel approx}.


To proceed, we will now choose $j= 2$. By combining the above bounds, we find that as $t\to 0^\limplus$, 
\begin{equation}\label{eq: trace estimate}
\begin{aligned}
	\Tr(&e^{t\Delta_\Omega^{(\sigma)}})- \Tr(e^{t\Delta_\Omega^{(0)}})\\
	&=
	\int_\Omega(K_2(t, x, x)-k_\Omega^{(0)}(t, x, x))\,dx + O(t^{-\frac{d-3}{2}- \frac{d-1}{p}})\\
	&=
	 - \int_\Omega\int_0^t\int_{\partial\Omega}k_\Omega^{(0)}(t-s, x, y)k_\Omega^{(0)}(s, y, x)\sigma(y)d\Haus^{d-1}(y)dsdx + O(t^{-\frac{d-3}{2}- \frac{d-1}{p}}) \,.
\end{aligned}
\end{equation}
Using this bound we are now able to prove leading order asympotics for the difference of the Robin and Neumann heat traces.

\begin{theorem}\label{thm: asymptotics heat difference}
	Let $\Omega \subset \R^d$ be open, bounded, with Lipschitz regular boundary. Assume that $\sigma\in L^p(\partial\Omega)$ for some $p>2(d-1)$ and that Assumption \ref{gaussianboundass} is satisfied. Then, 
	\begin{equation*}
		\Tr(e^{t\Delta_\Omega^{(0)}})- \Tr(e^{t\Delta_\Omega^{(\sigma)}}) = \frac{1}{2\pi (4\pi t)^{\frac{d-2}2}}\int_{\partial\Omega}\sigma(x)\,d\Haus^{d-1}(x)+ o(t^{-\frac{d-2}2})\quad \mbox{as }t\to 0^\limplus\,.
	\end{equation*} 		
\end{theorem}	

For sets $\Omega$ whose boundary is sufficiently smooth the above theorems can be deduced from three-term asymptotic expansions for $\Tr(e^{t\Delta_\Omega^{(\sigma)}})$; see, e.g., \cite{BransonGilkey_90}. As we have argued above, under our weak assumptions such three-term asymptotics for the individual traces are not necessarily true.

\begin{proof}
	Since $p>2(d-1)$, we have $t^{-\frac{d-3}{2}- \frac{d-1}{p}} = o(t^{- \frac{d-2}{2}})$ as $t\to 0^\limplus$. Thus, to complete the proof it remains to compute the asymptotic behavior of the triple integral in the right-hand side of~\eqref{eq: trace estimate}.

	By Fubini's theorem and the semigroup property of the heat kernel we have
	\begin{align*}
		\int_\Omega\int_0^t &\int_{\partial\Omega}k_\Omega^{(0)}(t-s, x, y)k_\Omega^{(0)}(s, y, x)\sigma(y)\,d\Haus^{d-1}(y)dsdx\\
		&= 
		\int_0^t \int_{\partial\Omega}\biggl(\int_\Omega k_\Omega^{(0)}(t-s, x, y)k_\Omega^{(0)}(s, y, x)\,dx\biggr)\sigma(y)\,d\Haus^{d-1}(y)ds\\
		&=
		\int_0^t \int_{\partial\Omega}k_\Omega^{(0)}(t, y, y)\sigma(y)\,d\Haus^{d-1}(y)ds\\
		&=
		t \int_{\partial\Omega}k_\Omega^{(0)}(t, y, y)\sigma(y)\,d\Haus^{d-1}(y)\,.
	\end{align*}
    Note that here we used the semigroup property pointwise on $\partial\Omega$, which follows from the continuity of the Neumann heat kernel up to the boundary; see e.g., \cite[Corollary 3.2]{Brown93} or \cite[Corollary 9.4]{FrankLarson_24}. (The latter proof is for the convex setting but it only relies on the validity of a pointwise bound on the heat kernel.)
    
	For any $\epsilon>0$ sufficiently small, there exists a set $G \subset \partial\Omega$ such that
	\begin{equation}\label{eq: estimate on good set}
		\Haus^{d-1}(\partial\Omega \setminus G) \leq \epsilon \Haus^{d-1}(\partial\Omega) \quad \mbox{and}\quad |(4\pi t)^{\frac{d}2}k_\Omega^{(0)}(t, y, y)-2|\lesssim_\Omega \epsilon \quad \mbox{for all }y \in G
	\end{equation}
	and all $t \lesssim_\Omega \epsilon$. Indeed, let $G$ be the $(\epsilon, r)$-good set constructed in \cite[Section 4]{Brown93} with $r>0$ chosen so small that the first estimate in~\eqref{eq: estimate on good set} holds. Let $\mathcal{G}\subset \Omega$ be the associated sawtooth region \cite[eq. (1.2)]{Brown93}. Note that by definition $G$ is the subset of $\partial\Omega$ reachable as limits of points in $\mathcal{G}$. The second estimate in \eqref{eq: estimate on good set}  follows by applying \cite[Proposition~2.1 (2.2+)]{Brown93} and taking the limit $\mathcal{G}\ni x \to y\in G$.
    
    Consequently, using also that $|k^{(0)}_\Omega(t, x, x)|\lesssim_{\Omega} t^{-\frac{d}2}$ for all $x \in \Omega$,
	\begin{align*}
		\int_{\partial\Omega}k_\Omega^{(0)}(t, y, y)\sigma(y)\,d\Haus^{d-1}(y)
		&= 2(4\pi t)^{-\frac{d}2}\int_{\partial\Omega} \sigma(y)\,d\Haus^{d-1}(y)\\
		&\quad  + O\Bigl(t^{-\frac{d}2}\Bigl(\epsilon + \int_{\partial\Omega \setminus G}|\sigma(y)|d\Haus^{d-1}(y)\Bigr)\Bigr)\\
		&= 2(4\pi t)^{-\frac{d}2}\int_{\partial\Omega} \sigma(y)\,d\Haus^{d-1}(y)  + o_{\epsilon \to 0^\limplus}(t^{-\frac{d}2})\,.
	\end{align*}
	
	Putting the above estimates together we conclude that
	\begin{align*}
	\Tr(e^{t\Delta_\Omega^{(0)}})&- \Tr(e^{t\Delta_\Omega^{(\sigma)}})\\
	&=
	 2t(4\pi t)^{-\frac{d}2}\int_{\partial\Omega} \sigma(y)\,d\Haus^{d-1}(y) + o_{\epsilon \to 0^\limplus}(t^{-\frac{d-2}2})+ o_{t\to 0^\limplus}(t^{-\frac{d-2}{2}})\,.
	\end{align*}
	The proof is concluded by sending $t\to 0^\limplus$ and then $\epsilon \to 0^\limplus$.
\end{proof}


\subsection{Density asymptotics}

In this subsection we consider the spectral function
$$
\rho^{(\sigma)}_\Omega(\lambda,x) := (-\Delta_\Omega^{(\sigma)}-\lambda)_\limminus^0(x,x) = \1(-\Delta_\Omega^{(\sigma)}<\lambda)(x,x)
\qquad\text{for}\ x\in\overline\Omega\,.
$$

\begin{theorem}\label{density}
    Let $\Omega \subset \R^d$ be open, bounded, with Lipschitz regular boundary. Assume that $\sigma\in L^p(\partial\Omega)$ for some $p>d-1$ and that Assumption \ref{gaussianboundass} holds. Then
    $$
    \lambda^{-d/2} \rho^{(\sigma)}_\Omega(\lambda,y) \to 2 L_{0,d}^{\rm sc}
    \qquad\text{for}\ \mathcal H^{d-1}-\text{a.e.}\ y\in\partial\Omega \,.
    $$
    Moreover, for any $\lambda_0>0$ one has
    $$
    \sup_{\lambda\geq\lambda_0,\ x\in\overline{\Omega}} \lambda^{-d/2} \rho^{(\sigma)}_\Omega(\lambda,x) <\infty \,.
    $$
\end{theorem}

Of course, combining the first and second items in the theorem with the dominated convergence theorem, we infer that $\lambda^{-d/2} \rho^{(\sigma)}_\Omega(\lambda,\cdot) \to 2 L_{0,d}^{\rm sc}$ in $L^q(\partial\Omega)$ for any $q<\infty$.

Theorem \ref{density} quantifies the phenomenon of boundary concentration that occurs in the case of Robin boundary conditions. In fact, the asymptotics on the boundary should be compared with the bulk asymptotics 
$$
\lambda^{-d/2} \rho^{(\sigma)}_\Omega(\lambda,x) \to L_{0,d}^{\rm sc}
\qquad\text{for all}\ x\in\Omega \,.
$$
Thus the limit on the boundary is twice the limit in the bulk.

\begin{proof}
    The facts about the Neumann heat kernel $k_\Omega^{(0)}$ recalled in the previous proof imply that
    $$
    t^\frac d2 k_\Omega^{(0)}(t,y,y) \to 2 (4\pi)^{-\frac d2}
    \qquad\text{for}\ \mathcal H^{d-1}-\text{a.e.}\ y\in\partial\Omega \,. 
    $$
    Using the bound from Proposition \ref{prop: improved kernel approx} with $j=1$ we deduce the same pointwise a.e.-convergence for the Robin heat kernel $k_\Omega^{(\sigma)}$. Since
    $$
    k^{(\sigma)}_\Omega(t,y,y) = \int_0^\infty e^{-t\lambda} \,d\rho_\Omega^{(\sigma)}(\lambda,y) \,,
    $$
    the claimed pointwise a.e.-convergence of $\rho_\Omega^{(\sigma)}$ follows from the standard Tauberian theorem for the Laplace transform (see, e.g., \cite[Theorem 10.3]{Simon_FunctionalIntegrationBook} or \cite[Theorem VII.3.2]{Korevaar_TauberianTheory_book}).

    To obtain the claimed uniform bound, we observe that, by \eqref{eq: Assumed Gaussian bound},
    $$
    \rho_\Omega^{(\sigma)}(\lambda,x) \leq e^{t\lambda} k_\Omega^{(\sigma)}(t,x,x) \leq M_\sigma \max\{1,(t/t_0)^{\frac d2}\} t^{-\frac d2} e^{(\Lambda+\lambda) t}
    \qquad\text{for all}\ x\in\overline{\Omega} \,.
    $$
    Choosing $t=\lambda^{-1}$, we obtain the claimed bound.
\end{proof}


\subsection{Differences of Riesz means}

We now apply a technique from \cite{FrankLarson_24} to deduce Riesz means asymptotics from heat trace asymptotics.

\begin{theorem}\label{thm: Difference of Riesz means}
	Let $\Omega \subset \R^d$ be open, bounded, with Lipschitz regular boundary. Assume that $\sigma \in L^p(\partial\Omega)$ for some $p>2(d-1)$ and that Assumption \ref{gaussianboundass} is satisfied. Then for any $\gamma \geq 1$, as $\lambda \to \infty$,
	\begin{equation*}
		\Tr(-\Delta_\Omega^{(0)}-\lambda)_\limminus^\gamma- \Tr(-\Delta_\Omega^{(\sigma)}-\lambda)_\limminus^\gamma = \frac{L_{\gamma,d-2}^{\rm sc}}{2\pi} \int_{\partial\Omega}\sigma(x)\,d\Haus^{d-1}(x)\, \lambda^{\gamma+ \frac{d-2}{2}}+ o(\lambda^{\gamma+ \frac{d-2}{2}})\,.
	\end{equation*} 
\end{theorem}

Before proving Theorem~\ref{thm: Difference of Riesz means} let us argue that it implies Theorem~\ref{thm: Difference of Riesz means intro}.

\begin{proof}[Proof of Theorem \ref{thm: Difference of Riesz means intro}]
    By Proposition \ref{gaussianbound}, Assumption \ref{gaussianboundass} is satisfied for $\sigma$ as in Theorem \ref{thm: Difference of Riesz means intro}. Therefore, Theorem~\ref{thm: Difference of Riesz means intro} is a direct consequence of Theorem~\ref{thm: Difference of Riesz means}.
\end{proof}

\begin{proof}[Proof of Theorem~\ref{thm: Difference of Riesz means}]
    By Proposition \ref{gaussianbound} (or by \cite{Gesztesy_etal_PAMS15}), the kernel $k_\Omega^{(\sigma_\limplus)}$ satisfies a Gaussian bound as in \eqref{eq: Assumed Gaussian bound}. By Theorem~\ref{thm: asymptotics heat difference},
    \begin{align*}
        \Tr(e^{t\Delta_\Omega^{(0)}})- \Tr(e^{t\Delta_\Omega^{(\sigma)}}) &= \frac{1}{2\pi (4\pi t)^{\frac{d-2}2}}\int_{\partial\Omega}\sigma(x)\,d\Haus^{d-1}(x)(1+o(1))\quad \mbox{as }t\to 0^\limplus\,,\\
        \Tr(e^{t\Delta_\Omega^{(0)}})- \Tr(e^{t\Delta_\Omega^{(\sigma_\limplus)}}) &= \frac{1}{2\pi (4\pi t)^{\frac{d-2}2}}\int_{\partial\Omega}\sigma_\limplus(x)\,d\Haus^{d-1}(x)(1+o(1))\quad \mbox{as }t\to 0^\limplus\,,
    \end{align*}
    and therefore, by linearity,
    \begin{align*}
        \Tr(e^{t\Delta_\Omega^{(\sigma_\limplus)}})- \Tr(e^{t\Delta_\Omega^{(\sigma)}}) &= \frac{1}{2\pi (4\pi t)^{\frac{d-2}2}}\int_{\partial\Omega} \!(\sigma(x) \!-\! \sigma_\limplus(x))\,d\Haus^{d-1}(x)(1+o(1))\ \mbox{as }t\to 0^\limplus.
    \end{align*}

	Consider the functions
	\begin{align*}
	f_1(\lambda) := \Tr(-\Delta_\Omega^{(\sigma)}-\lambda)_\limminus - \Tr(-\Delta_\Omega^{(\sigma_\limplus)}-\lambda)_\limminus\,, \\
    f_2(\lambda) := \Tr(-\Delta_\Omega^{(0)}-\lambda)_\limminus - \Tr(-\Delta_\Omega^{(\sigma_\limplus)}-\lambda)_\limminus \,.
	\end{align*}
	We begin by proving that $f_j$ is nondecreasing for $j=1, 2$. For definiteness, we write out the argument for $j=1$, the argument being identical for $j=2$. For any lower semibounded operator $H$ we have $\Tr(H-\lambda)_\limminus = \int_{-\infty}^\lambda \Tr(H-\mu)_\limminus^0\,d\mu$, so in particular
    $$
	f_1(\lambda) = \int_{-\infty}^\lambda \left( \Tr(-\Delta_\Omega^{(\sigma)}-\mu)_\limminus^0 - \Tr(-\Delta_\Omega^{(\sigma_\limplus)}-\mu)_\limminus^0 \right)d\mu \,.
	$$
	When $\lambda_1\leq \lambda_2$ it follows that
	$$
	f_1(\lambda_2) - f_1(\lambda_1) = \int_{\lambda_1}^{\lambda_2} \left( \Tr(-\Delta_\Omega^{(\sigma)}-\mu)_\limminus^0 - \Tr(-\Delta_\Omega^{(\sigma_\limplus)}-\mu)_\limminus^0 \right)d\mu \,.
	$$
	Since $\sigma \leq \sigma_\limplus$ the variational principle implies that $\Tr(-\Delta_\Omega^{(\sigma)}-\mu)_\limminus^0 - \Tr(-\Delta_\Omega^{(\sigma_\limplus)}-\mu)_\limminus^0\geq 0$ for any $\mu$, so we deduce the claimed monotonicity of $f_1$.
	
	We note that
	\begin{align*}
	\int_0^\infty e^{-t\lambda} df_1(\lambda) = t^{-1} \bigl( \Tr(e^{t\Delta_\Omega^{(\sigma)}}) - \Tr(e^{t\Delta_\Omega^{(\sigma_\limplus)}}) \bigr)\,\\
    \int_0^\infty e^{-t\lambda} df_2(\lambda) = t^{-1} \bigl( \Tr(e^{t\Delta_\Omega^{(0)}}) - \Tr(e^{t\Delta_\Omega^{(\sigma_\limplus)}}) \bigr)\,.
	\end{align*}
	As explained above, Theorem \ref{thm: asymptotics heat difference} implies that
	\begin{align*}
	\int_0^\infty e^{-t\lambda} df_1(\lambda) &= \frac{1}{2\pi t (4\pi t)^{\frac{d-2}2}} \int_{\partial\Omega}(\sigma_\limplus(x)-\sigma(x))\,d\Haus^{d-1}(x) (1 + o(1))\,,\\
    \int_0^\infty e^{-t\lambda} df_2(\lambda) &= \frac{1}{2\pi t (4\pi t)^{\frac{d-2}2}} \int_{\partial\Omega}\sigma_\limplus(x)\,d\Haus^{d-1}(x) (1 + o(1))\,,
	\end{align*}
    as $t\to 0^\limplus$.
	It follows from the standard Tauberian theorem for the Laplace transform (see, e.g., \cite[Theorem 10.3]{Simon_FunctionalIntegrationBook} or \cite[Theorem VII.3.2]{Korevaar_TauberianTheory_book}) that
	\begin{align*}
	f_1(\lambda) &= \frac{1}{2\pi \Gamma(1+\frac{d}2) (4\pi)^{\frac{d-2}2}} \int_{\partial\Omega}(\sigma_\limplus(x)-\sigma(x))\,d\Haus^{d-1}(x) \lambda^{\frac{d}{2}} (1 + o(1))\,,\\
    f_2(\lambda) &= \frac{1}{2\pi \Gamma(1+\frac{d}2) (4\pi)^{\frac{d-2}2}} \int_{\partial\Omega}\sigma_\limplus(x)\,d\Haus^{d-1}(x) \lambda^{\frac{d}{2}} (1 + o(1))\,,
	\end{align*}
	as $\lambda \to \infty$. Since $(4\pi)^{-\frac{d-2}{2}}\Gamma(\frac{d}{2}+1)^{-1} = L_{1,d-2}^{\rm sc}$, the proof for $\gamma=1$ is completed by noticing that
    $$
        \Tr(-\Delta_\Omega^{(0)}-\lambda)_\limminus - \Tr(-\Delta_\Omega^{(\sigma)}-\lambda)_\limminus = f_2(\lambda) - f_1(\lambda)\,.
    $$
    
    The asymptotics for $\gamma>1$ are deduced from those for $\gamma=1$ by integration with respect to $\lambda$ (see \cite{AizenmanLieb} or \cite[Subsection 5.1.1]{FrankLaptevWeidl}).
\end{proof}


\subsection{Neumann--Robin gap asymptotics}

As a final topic we now turn to the asymptotics of the Neumann--Robin gaps $\lambda_n(-\Delta_\Omega^{(\sigma)})-\lambda_n(-\Delta_\Omega^{(0)})$.

\begin{corollary}\label{cor: gap asymptotics}
    Let $\Omega \subset \R^d$ be open, bounded, with Lipschitz regular boundary. Assume that $\sigma \in L^p(\partial\Omega)$ for some $p>2(d-1)$ and that Assumption \ref{gaussianboundass} is satisfied. Then, as $N \to \infty,$
	\begin{equation*}
		\sum_{n=1}^N \frac{\lambda_n(-\Delta_\Omega^{(\sigma)})-\lambda_n(-\Delta_\Omega^{(0)})}{N} = \frac{2}{|\Omega|}\int_{\partial\Omega}\sigma(x)\,d\Haus^{d-1}(x)+ o(1)\,.
	\end{equation*} 
\end{corollary}

\begin{proof}[Proof of Corollary \ref{cor: gap asymptotics intro}]
    By Proposition \ref{gaussianbound}, Assumption \ref{gaussianboundass} is satisfied for $\sigma$ as in Corollary \ref{cor: gap asymptotics intro}. Therefore, Corollary~\ref{cor: gap asymptotics intro} is a direct consequence of Corollary~\ref{cor: gap asymptotics}.
\end{proof}

A key ingredient in our proof of this corollary is the following lemma.
\begin{lemma}\label{lem: gap and Riesz diff equiv}
    Let $\{a_n\}_{n\geq 1}$ and $\{b_n\}_{n\geq 1}$ be nondecreasing sequences. Assume that there are constants $\alpha>0$, $C>0$ such that, as $n\to\infty$,
    \begin{equation}\label{eq: assumed sequence asymptotics}
    a_n = C n^\alpha + o(n^\alpha)\,.
    \end{equation}
    Then
    \begin{equation}\label{eq: assumed Riesz diff asymptotics}
    \sum_{n\geq 1} (a_n - \lambda)_\limminus - \sum_{n\geq 1} (b_n - \lambda)_\limminus = D \lambda^{\frac{1}{\alpha}} + o(\lambda^{\frac{1}{\alpha}}) \quad \mbox{as } \lambda \to \infty\,,
    \end{equation}
    for some $D \in \R$, if and only if
    \begin{equation}\label{eq: assumed gap asymptotics}
    \sum_{n=1}^N \frac{b_n - a_n}{N} = C^{\frac{1}{\alpha}}D + o(1) \quad \mbox{as } N \to \infty\,.
    \end{equation}
\end{lemma}

Before proving Lemma~\ref{lem: gap and Riesz diff equiv} let us show how it can be combined with Theorem \ref{thm: Difference of Riesz means} to imply Corollary \ref{cor: gap asymptotics}.

\begin{proof}[Proof of Corollary \ref{cor: gap asymptotics}]
    By Weyl's law
    \begin{equation}\label{eq: Weyls law}
        \lambda_n(-\Delta_\Omega^{(0)}) = \Bigl(\frac{n}{L_{0,d}^{\rm sc}|\Omega|}\Bigr)^{\frac{2}d}+o(n^{\frac{2}d}) \quad \mbox{as } n\to \infty\,,
    \end{equation}
    and by Theorem \ref{thm: Difference of Riesz means}
    \begin{equation}\label{eq: diff 2}
    \begin{aligned}
        \sum_{n\geq1}(\lambda_n(-\Delta_\Omega^{(0)})&-\lambda)_\limminus-\sum_{n\geq1}(\lambda_n(-\Delta_\Omega^{(\sigma)})-\lambda)_\limminus\\
        &=
        \Tr(-\Delta_\Omega^{(0)}-\lambda)_\limminus- \Tr(-\Delta_\Omega^{(\sigma)}-\lambda)_\limminus\\
        &=\frac{L_{1,d-2}^{\rm sc}}{2\pi} \int_{\partial\Omega}\sigma(x)\,d\Haus^{d-1}(x) \lambda^{\frac{d}2}+o(\lambda^{\frac{d}2}) \quad \mbox{as } \lambda \to \infty\,.
	\end{aligned} 
    \end{equation}
    Therefore, by Lemma~\ref{lem: gap and Riesz diff equiv} with 
    $$\alpha = \frac{2}{d}\,,\quad  C= (L_{0,d}^{\rm sc}|\Omega|)^{-\frac{2}d}\,,\quad \mbox{and } D = \frac{L_{1,d-2}^{\rm sc}}{2\pi} \int_{\partial\Omega}\sigma(x)\,d\Haus^{d-1}(x)$$ 
    we deduce that
    \begin{equation*}
        \lim_{N\to \infty}\sum_{n=1}^N\frac{\lambda_n(-\Delta_\Omega^{(\sigma)})-\lambda_n(-\Delta_\Omega^{(0)})}{N} =\frac{L_{1,d-2}^{\rm sc}}{2\pi L_{0,d}^{\rm sc}|\Omega|} \int_{\partial\Omega}\sigma(x)\,d\Haus^{d-1}(x)
    \end{equation*}
    The fact that $\frac{L_{1,d-2}^{\rm sc}}{2\pi L_{0,d}^{\rm sc}} = 2$ completes the proof.
\end{proof}

\begin{proof}[Proof of Lemma~\ref{lem: gap and Riesz diff equiv}]
    We first note that by the assumption on $\{a_n\}_{n\geq 1}$ we have
    \begin{equation*}
        \sum_{n\geq 1}(\lambda-a_n)_\limminus = \frac{\alpha}{1+\alpha}C^{-\frac{1}\alpha}\lambda^{1+\frac{1}\alpha}+o(\lambda^{1+\frac{1}\alpha})\quad \mbox{as } \lambda \to \infty
    \end{equation*}
    and
    \begin{equation*}
        N^{-1}\sum_{n=1}^N a_n  = \frac{C}{1+\alpha}N^\alpha +o(N^\alpha) \quad \mbox{as }N \to \infty\,. 
    \end{equation*}
    Consequently, if \eqref{eq: assumed Riesz diff asymptotics} holds then
    \begin{equation*}
        \sum_{n\geq 1}(\lambda-b_n)_\limminus = \frac{\alpha}{1+\alpha}C^{-\frac{1}\alpha}\lambda^{1+\frac{1}\alpha}+o(\lambda^{1+\frac{1}\alpha})\quad \mbox{as } \lambda \to \infty\,,
    \end{equation*}
    while if instead \eqref{eq: assumed gap asymptotics} holds then it also holds that
    \begin{equation*}
        N^{-1}\sum_{n=1}^N b_n  = \frac{C}{1+\alpha}N^\alpha +o(N^\alpha) \quad \mbox{as }N \to \infty\,. 
    \end{equation*}
    (As is well-known, these two asymptotic statements are in fact equivalent; a proof can be given, for instance, along the lines of \cite[Lemma A.1]{FrankGeisinger16}.) In either case it follows by a simple Tauberian-type argument that
    \begin{equation}\label{eq: second sequence asymptotics}
        b_n = C n^\alpha + o(n^\alpha) \quad \mbox{as } n \to \infty\,,
    \end{equation} 
    see, e.g., \cite[Lemma I.17.1]{Korevaar_TauberianTheory_book}.

    We next note that for any $N$,
    \begin{align*}
        N^{-1} \sum_{n=1}^N (b_n - a_n) & = N^{-1} \Biggl( \sum_{n=1}^N (b_n - b_N) - \sum_{n=1}^N (a_n - b_N) \Biggr) \\
        & \leq N^{-1} \Biggl( - \sum_{n=1}^N (b_n - b_N)_\limminus + \sum_{n=1}^N (a_n - b_N)_\limminus \Biggr) \\
        & \leq N^{-1} \Biggl( - \sum_{n\geq 1} (b_n - b_N)_\limminus + \sum_{n\geq 1} (a_n - b_N)_\limminus \Biggr) \\
        & = \frac{b_N^{\frac{1}\alpha}}{N} \ \frac{- \sum (b_n - b_N)_\limminus + \sum (a_n - b_N)_\limminus}{b_N^{\frac{1}\alpha}}
    \end{align*}
    and
    \begin{align*}
        N^{-1} \sum_{n=1}^N (b_n - a_n) & = N^{-1} \Biggl( \sum_{n=1}^N (b_n - a_N) - \sum_{n=1}^N (a_n - a_N) \Biggr) \\
        & \geq N^{-1} \Biggl( - \sum_{n=1}^N (b_n - a_N)_\limminus + \sum_{n=1}^N (a_n - a_N)_\limminus \Biggr) \\
        & \geq N^{-1} \Biggl( - \sum_{n\geq 1} (b_n - a_N)_\limminus + \sum_{n\geq 1} (a_n - a_N)_\limminus \Biggr) \\
        & = \frac{a_N^{\frac{1}\alpha}}{N} \ \frac{- \sum (b_n - a_N)_\limminus + \sum (a_n - a_N)_\limminus}{a_N^{\frac{1}\alpha}} \,.
    \end{align*}
    Writing $\{c_n\}_{n\geq 1}$ for either $\{a_n\}_{n\geq 1}$ or $\{b_n\}_{n\geq 1}$ we have, by assumption
    $$
    \frac{- \sum (b_n - c_N)_\limminus + \sum (a_n - c_N)_\limminus}{c_N^{\frac{1}\alpha}}
    \to D
    $$
    and, by \eqref{eq: assumed sequence asymptotics} and \eqref{eq: second sequence asymptotics},
    $$
    N^{-1}c_N^{\frac{1}\alpha} \to C^{\frac{1}\alpha} \,.
    $$
    Thus, the upper and the lower bound coincide asymptotically and we obtain that
    \begin{equation*}
        N^{-1} \sum_{n=1}^N (b_n - a_n) \to C^{\frac{1}\alpha}D\,.
    \end{equation*}
    We have thus proved the first implication.

    To prove the second implication we argue similarly. For $\lambda > \max\{a_1, b_1\}$ define
    \begin{equation*}
        N_a(\lambda) = \#\{n \in \N: a_n \leq \lambda\} \quad \mbox{and} \quad N_b(\lambda) = \#\{n \in \N: b_n \leq \lambda\}\,.
    \end{equation*}
    By the asymptotics of the sequences $\{a_n\}_{n\geq 1}, \{b_n\}_{n\geq 1}$ it follows that $N_a(\lambda) \sim \bigl(\frac{\lambda}{C}\bigr)^{\frac{1}{\alpha}}$ and $N_b(\lambda) \sim \bigl(\frac{\lambda}{C}\bigr)^{\frac{1}{\alpha}}$ as $\lambda \to \infty$.

    Using $x_\limminus = \max\{0, -x\} \geq -x$ it holds that, for any $\lambda >\max\{a_1, b_1\}$,
    \begin{align*}
        \lambda^{-\frac{1}\alpha}\sum_{n\geq 1}((a_n-\lambda)_\limminus-(b_n-\lambda)_\limminus)
        &\leq
        \lambda^{-\frac{1}\alpha}\sum_{n= 1}^{N_a(\lambda)}((a_n-\lambda)_\limminus-(b_n-\lambda)_\limminus)\\
        &\leq
        \lambda^{-\frac{1}\alpha}\sum_{n= 1}^{N_a(\lambda)}(b_n-a_n)\\
        &=
        \frac{N_a(\lambda)}{\lambda^{\frac{1}\alpha}}\frac{1}{N_a(\lambda)}\sum_{n= 1}^{N_a(\lambda)}(b_n-a_n)
    \end{align*}
    and
    \begin{align*}
        \lambda^{-\frac{1}\alpha}\sum_{n\geq 1}((a_n-\lambda)_\limminus-(b_n-\lambda)_\limminus)
        &\geq 
        \lambda^{-\frac{1}\alpha}\sum_{n= 1}^{N_b(\lambda)}((a_n-\lambda)_\limminus-(b_n-\lambda)_\limminus)\\
        &\geq
        \lambda^{-\frac{1}\alpha}\sum_{n= 1}^{N_b(\lambda)}(b_n-a_n)\\
        &=
        \frac{N_b(\lambda)}{\lambda^{\frac{1}\alpha}}\frac{1}{N_b(\lambda)}\sum_{n= 1}^{N_b(\lambda)}(b_n-a_n)\,.
    \end{align*}
    The proof of the second implication can now be completed in the same manner as that of the first, relying on the asymptotics in~\eqref{eq: assumed gap asymptotics} and the asymptotics for $N_a(\lambda), N_b(\lambda)$ computed above.
\end{proof}


\bibliographystyle{amsalpha}

\end{document}